\documentclass{amsart}

\usepackage{appendix}
\usepackage{amssymb}

\usepackage{thmtools}
\usepackage{thm-restate}
\usepackage{hyperref}
\usepackage{cleveref}

\usepackage{mathrsfs}

\usepackage{graphicx}

\usepackage[cmtip,all]{xy}

\usepackage{amsthm}
\newtheorem{theorem}{Theorem}[section]

\newtheorem{lemma}[theorem]{Lemma}
\newtheorem{corollary}[theorem]{Corollary}
\newtheorem{proposition}[theorem]{Proposition}
\newtheorem*{conjecture*}{Conjecture}
\newtheorem{introthm}{Theorem}

\newtheorem{introcor}[introthm]{Corollary}

\theoremstyle{definition}
\newtheorem{definition}[theorem]{Definition} 

\newtheorem{remark}[theorem]{Remark}
\newtheorem*{question*}{Question}

\newcommand{\Rbb}{{\mathbb R}}
\newcommand{\Zbb}{{\mathbb Z}}

\newcommand{\Fbb}{{\mathbb F}}
\newcommand{\Ebb}{{\mathbb E}}
\newcommand{\Sbb}{{\mathbb S}}

\makeatletter
\def\l@subsection{\@tocline{2}{0pt}{2.5pc}{5pc}{}}
\makeatother

\usepackage[style=alphabetic,
maxbibnames=10]{biblatex}

\usepackage{booktabs}  
\usepackage{adjustbox}
\usepackage{tikz}
\usepackage{quiver}
\usepackage{hyperref}
\usepackage{mathtools}
\usepackage{enumitem}
\usepackage{soul}

\numberwithin{equation}{section}

\begin{document}

\title{$H\mathbb{Z}/4$ is not an $\mathbb{E}_2$-Thom Spectrum over the $2$-Complete Sphere Spectrum}


\author{Mattie Ji}
\address{University of Pennsylvania}
\curraddr{}
\email{mji13@sas.upenn.edu}
\thanks{}


\date{}

\dedicatory{}

\maketitle


\begin{abstract}
    For any prime number $p$, the classic Hopkins-Mahowald theorem asserts that $H\mathbb{Z}/p$ is an $\mathbb{E}_2$-Thom spectrum over the $p$-complete sphere spectrum. Kitchloo extended this result to show that $H\mathbb{Z}/p^k$ is an $\mathbb{E}_2$-Thom spectrum over the $p$-local sphere spectrum, except for the case when $(p, k) = (2, 2)$. We solve the remaining case by showing that $H\mathbb{Z}/4$ is not an $\mathbb{E}_2$-Thom spectrum over the $2$-complete sphere spectrum, and consequently the $2$-local sphere spectrum. This work was developed with substantial assistance from GPT-5.6 Sol and GPT-6 Astra, and an early draft was subsequently reviewed using Claude Fable 5.1.
\end{abstract}

\section{Introduction}

For any prime number $p$, the classic Hopkins-Mahowald theorem showed that $H\Zbb/p$ is an $\mathbb{E}_2$-Thom spectrum over the $p$-complete sphere spectrum $\Sbb^{\wedge}_p$. The case for $p = 2$ was originally due to Mahowald \cite{MAHOWALD1977249,Mahowald1979}, and Hopkins observed the case generalizes for $p > 2$.

Since then, there has been many interests in constructing various versions of Eilenberg-MacLane spectra as certain Thom spectra with structures. To name a few, the Hopkins-Mahowald theorem has been reproved in \cite{Mathew_2015} (Theorem 4.16) and \cite{Antol_n_Camarena_2018} (Section 5.1). Mao \cite{Mao_2023} constructed the Eilenberg-MacLane spectra of perfectoid rings as $\mathbb{E}_2$-Thom spectra over spherical Witt vectors. In the equivariant case, Behrens and Wilson \cite{Behrens2017AA} showed that the $C_2$-equivariant Eilenberg-MacLane spectrum associated to the constant $C_2$-Mackey functor $\underline{\Fbb_2}$ can be constructed as an analogous Thom spectrum, and this was generalized by Hahn and Wilson \cite{Hahn_2020} to any finite $p$-power cyclic group $G$. This was further generalized by Levy \cite{https://doi.org/10.1112/topo.12230} to any faithful representation of a $p$-group into $O(2)$.

The work of most interest to this paper is that of Kitchloo. In \cite{Kitchloo2020}, Kitchloo constructed an $\mathbb{E}_2$-Thom spectrum structure on $H\Zbb/p^k$ over the $p$-local sphere spectrum $\Sbb_{(p)}$ for all choices of $(p, k)$ except for the case when $(p, k) = (2, 2)$. To summarize the status so far:

\begin{theorem}[Hopkins-Mahowald, Kitchloo]\label{thm::known}
Let $p$ be a prime number and $k \geq 1$. For all choices of $(p, k) \neq (2, 2)$, $H\Zbb/p^k$ is an $\mathbb{E}_2$-Thom spectrum over $\Sbb_{(p)}$, and hence over $\Sbb^{\wedge}_p$.
\end{theorem}

Given the status of Theorem~\ref{thm::known}, a natural question would be whether we could extend the construction to the last remaining case for $H\Zbb/4$. In this work, we show that the remaining case is not possible.

\begin{introthm}\label{thm::main}
    $H\Zbb/4$ is not an $\mathbb{E}_2$-Thom spectrum over $\mathbb{S}_{(2)}$ or $\mathbb{S}_{2}^{\wedge}$.
\end{introthm}

Combining Theorem~\ref{thm::main} with Theorem~\ref{thm::known}, we have the following corollary.
\begin{introcor}
The spectrum $H\Zbb/p^k$ is an $\mathbb{E}_2$-Thom spectrum over $\mathbb{S}^{\wedge}_p$ if and only if $(p, k) \neq (2, 2)$.
\end{introcor}

\begin{remark}
    It is possible to directly show that $H\Zbb/4$ is not a Thom spectrum at all over $\mathbb{S}$, see Section 9.2 of \cite{blumberg_einf}. It is therefore more interesting to pass to the $p$-local or $p$-complete case, as in the Hopkins-Mahowald theorem or Kitchloo's paper \cite{Kitchloo2020}.
\end{remark}

\begin{remark}
One motivation for constructing these Thom spectra is from topological Hochschild homology (THH). The main theorem of \cite{Blumberg_2010} gives a way to compute the THH of $\mathbb{E}_2$-Thom spectra. In the case of $H\Zbb/p$, this yields another proof of the classic Bökstedt periodicity theorem \cite{Bokstedt1986}. For instance, see \cite{KrauseNikolaus2018LecturesTHH}, \cite{Krause_Nikolaus_2022}, or \cite{cmak_lecture} for detailed explanations of this application. In the case of $H\Zbb/p^k$ for $(p, k) \neq (2, 2)$, the main result of \cite{Kitchloo2020} gave an alternative way to compute $\operatorname{THH}(\Zbb/p^k)$, which had previously been computed in \cite{Pirashvili01011995, BRUN200029}. Theorem~\ref{thm::main} rules out extending this approach to $H\Zbb/4$ using an $\Ebb_2$-Thom spectrum structure over $\Sbb^{\wedge}_2$. It is interesting to wonder whether a Thom spectrum approach in general could be developed to compute $\operatorname{THH}(\Zbb/4)$.
\end{remark}

\noindent \textbf{Proof and Content Outline.} In Section~\ref{sec::background}, we will introduce relevant background in Thom spectra, (co)homological suspensions, Dyer-Lashof operations, and transfer maps, which will be used in the proof of Theorem~\ref{thm::main}. Section~\ref{sec::background} also includes a detailed computation of a specific transfer map that will be important in the proof. Section~\ref{sec::reduction} reduces Theorem~\ref{thm::main} to the case over $\Sbb_{2}^{\wedge}$ with a path-connected space $X$.

The core of the proof is developed in Section~\ref{sec::construct_alpha} and Section~\ref{sec::failure}. In Section~\ref{sec::construct_alpha}, we assume for contradiction that $H\Zbb/4$ is an $\mathbb{E}_2$-Thom spectrum over $\Sbb^{\wedge}_2$ and construct a cohomology class $\alpha$. In Section~\ref{sec::failure}, we show simultaneously that $\alpha^2 \neq 0$ and $\alpha^2 = 0$, which will cause a contradiction. Finally, in Remark~\ref{rem::explain}, we explain conceptually why the proof of Theorem~\ref{thm::main} cannot be generalized to any other choice of $(p, k)$, so the case of $(p, k) = (2, 2)$ is truly anomalous.\\

\noindent \textbf{Acknowledgements.} This work was developed with substantial assistance from OpenAI’s ChatGPT, using GPT-5.6 Sol and GPT-6 Astra. These tools were used to explore and develop proof strategies, carry out mathematical computations, and identify relevant references. The author (MJ) critically assessed the proposed arguments, revised and simplified them, and carefully verified the proofs presented here. MJ wrote the initial manuscript in their own words and subsequently used these tools for mathematical and editorial review of successive drafts. MJ takes full responsibility for the final mathematical content and any errors.

An early draft was subsequently reviewed using Claude Fable 5.1. MJ thanks Yongxi (Aaron) Lin for providing access to Claude Fable 5.1 and sharing the resulting feedback for use in revising the manuscript.

MJ would like to thank Mark Behrens for helpful discussions on mathematical details in this work, Andrew Blumberg for helpful comments on an early draft, Preston Cranford for an inspiring conversation on a train that motivated this work, Nitu Kitchloo for helpful conversations on an early draft and suggestions to use Claude Fable 5.1, and Catherine Li for helpful conversations on power operations, transfers, and the Hopkins-Mahowald theorem. MJ is partially supported by the National Science Foundation Graduate Research Fellowship (DGE-2236662).

\section{Background}\label{sec::background}

In this section, we will cover some preliminary materials that will be used in the proof of Theorem~\ref{thm::main}.

\subsection{Thom Spectra and Units of Ring Spectra}

Let $X$ be a space and $R$ be, for simplicity, a (connective) $\mathbb{E}_{\infty}$-ring spectrum. We use $\mathfrak{gl}_1(R)$ to denote the associated spectrum of units, which has the property that $\pi_0(\mathfrak{gl}_1(R)) = \pi_0(R)^{\times}$ and $\pi_i(\mathfrak{gl}_1(R)) \cong \pi_i(R)$ for $i > 0$. We write $\operatorname{GL}_1(R)$ to denote the zeroth space $\Omega^{\infty} \mathfrak{gl}_1(R)$ and note that $\operatorname{BGL}_1(R)$ is the identity component in $\operatorname{Pic}(R)$.

\begin{definition}[Definition 1.4 of \cite{Ando_2013}, see also Definition 3.1 of \cite{Antol_n_Camarena_2018}]
Let $f: X \to \operatorname{Pic}(R)$ be a map (this is also called a local system of invertible $R$-modules over $X$). The \textit{Thom spectrum} of $f$ is
\[M_{R} f \coloneqq \operatorname{colim}(X \to \operatorname{Pic}(R) \to \operatorname{Mod}_{R}).\]
When the base ring $R$ is clear, we use $Mf$ to denote $M_{R} f$ instead. We say $Mf$ is an \textit{$\mathbb{E}_{n}$-Thom spectrum over $R$} if $X$ has an $\mathbb{E}_n$-structure such that $f$ is an $\mathbb{E}_n$-map. 
\end{definition}

The definition for an $\mathbb{E}_n$-Thom spectrum above is justified by the following.

\begin{theorem}[\cite{Lewis1978, LMS1986}, see also Corollary 3.2 of \cite{Antol_n_Camarena_2018}]
Suppose $X$ has an $\mathbb{E}_n$-structure such that $f$ is an $\mathbb{E}_n$-map, then $Mf$ is an $\mathbb{E}_n$-$R$-algebra.
\end{theorem}

We also employ the following version of Thom's isomorphism theorem.
\begin{theorem}[\cite{Lewis1978, LMS1986}, see also Lemma 3.15 and Proposition 3.16 of \cite{Antol_n_Camarena_2018}]\label{thm::thom_iso}
Suppose $X$ is grouplike and $Mf$ is an $\mathbb{E}_n$-Thom spectrum over $R$. Let $A$ be an $\mathbb{E}_{n+1}$-$R$-algebra, then any $\mathbb{E}_n$-$R$-algebra map $Mf \to A$ gives an equivalence of $\mathbb{E}_n$-$A$-algebras\[A \otimes_{R} Mf \simeq A \otimes \Sigma^{\infty}_+X.\]
\end{theorem}

\begin{remark}
Theorem~\ref{thm::thom_iso} is an $\mathbb{E}_n$-refinement of Corollary 2.26 of \cite{Ando_2013}.    
\end{remark}

It is implicit in the proof of Proposition 3.16 of \cite{Antol_n_Camarena_2018} and also explicit in Proposition 2.1.3 of \cite{Devalapurkar_2024} that the following change-of-rings formula for Thom spectra hold.
\begin{theorem}[Change-of-Rings Formula for Thom Spectra]\label{thm::change-of-rings}
    For simplicity, let $R \to R'$ be a map of $\mathbb{E}_{\infty}$-rings, then there is an equivalence
    \[M_{R'}(X \xrightarrow{f} \operatorname{Pic}(R) \to \operatorname{Pic}(R')) \cong M_{R}(X \xrightarrow{f} \operatorname{Pic}(R)) \otimes_{R} R'. \]
\end{theorem}

In nice cases, the Thom spectrum $Mf$ may admit a description in terms of the homotopy orbits of $R$ by some group action from $G$.
\begin{proposition}[Theorem 1.17 of \cite{Ando_2013}]\label{prop::quotient}
    Suppose $f$ factors through $\operatorname{BGL}_1(R)$ and $X = BG$, where $G$ is a group-like $A_{\infty}$-space. There is an action of $G$ on $R$ by $\Omega f: G \simeq \Omega BG \to \Omega \operatorname{BGL}_1(R) \simeq \operatorname{GL}_1(R)$ such that
    \[Mf \simeq R_{hG}.\]
\end{proposition}

Finally, we also record a proposition on the action of $\eta$ on $\mathfrak{gl}_1(R)$ and $R$.
\begin{proposition}[Proposition 4.1 of \cite{mathew2015torusactionsstablemodule}]\label{prop::eta_formula}
    Let $R$ be an $\mathbb{E}_{\infty}$-ring, $\eta$ be the stable Hopf map, and $\psi_i: \pi_i(R) \to \pi_i(\mathfrak{gl}_1(R))$ be the canonical isomorphism for $i \geq 1$. For all $x \in \pi_1(R)$, we have $\eta \cdot \psi_1(x) = \psi_{2}(\eta \cdot x + x^2)$.
\end{proposition}

\subsection{(Co)homological Suspension}

\begin{definition}
    Let $W$ be a based space. The counit of the suspension-loop adjunction yields a map $\Sigma \Omega W \to W$. This induces a \textit{cohomological suspension map}
    \[\sigma^*: \widetilde{H}^{i+1}(W; A) \to \widetilde{H}^i(\Omega W; A),\quad i \geq 0.\]
    If $W = BW'$, there is a \textit{homological suspension map $\sigma_*$}
    \[\sigma_*: \widetilde{H}_{q}(W') \to \widetilde{H}_{q+1}(BW').\]
\end{definition}

The following is a well-known fact in algebraic topology.
\begin{lemma}[See for Example, Chapter VII of \cite{Whitehead1978}] \label{lem::coh_map_loop}
Let $\alpha: W \to K(A, n)$ denote a map representing $\alpha \in \Tilde{H}^n(W, A)$. Then $\sigma^*(\alpha)$ represents the map
\[\Omega \alpha: \Omega W \to \Omega K(A, n) \simeq K(A, n-1).\]
\end{lemma}

\subsection{Dyer-Lashof Operations and Transfer Maps}

Here we briefly write down the Dyer-Lashof operations and all of the properties we will need in the proof. We refer the reader to \cite{lawson2020enringspectradyerlashof} for more details.

\begin{theorem}[Theorem 5.2, Remark 5.3, and Proposition 5.13 of \cite{lawson2020enringspectradyerlashof}]\label{thm::DL_operations}
Let $H = H\Fbb_2$. An $\mathbb{E}_n$-algebra $A_H$ in $\operatorname{Mod}_{H}$ has maps $Q^s: \pi_{m}(A_H) \to \pi_{m+s}(A_H)$ (called \textit{Dyer-Lashof operations}) for $m \leq s \leq m + n - 1$, such that
\[Q^{|x|}(x) = x^2,\]
and each $Q^s$ commutes with homology suspensions in $\Fbb_2$-homology of iterated loop spaces. Furthermore, if an $\mathbb{E}_n$-algebra structure on $A_{H} \in \operatorname{Mod}_{H}$ extends to an $\mathbb{E}_{n+1}$-algebra structure, then the operations $Q^i$ coincide with the previous ones.
\end{theorem}

We will also briefly introduce additive and multiplicative transfer maps for spectra. We refer the reader to Section 7 of \cite{15fd3722-af20-3529-b950-b5d19ca1d240} and Chapter VIII of \cite{h_infty_book} for more details.

\begin{definition} \label{def::transfer}
We let $\operatorname{tr}: \Sigma^{\infty}_+ BC_2 \to \Sigma^{\infty}_+ EC_2 \simeq \Sbb$ denote the stable transfer map with respect to the universal covering space $EC_2 \to BC_2$.
\begin{enumerate}
    \item For any spectrum $A$, precomposition gives a map, called the \textit{(additive) transfer map}
\[p^{A}_{!}: [\Sbb, A] \cong \pi_0(A) \to [\Sigma^{\infty}_+ BC_2, A] \cong A^0(BC_2).\]
This is clearly natural in $A$.
    \item If $A = R$ is an $\mathbb{E}_{\infty}$-ring, then there is also a \textit{multiplicative transfer map}
\[p^{R}_{\otimes}: \pi_0(R) \to R^0(BC_2).\]
\end{enumerate}
\end{definition}

Here we work through a specific example computing additive transfer maps that will be helpful for us later.

Consider the additive transfer map when $A = \Sbb$,
\[p_{!}^{\mathbb{S}}: \pi_0(\mathbb{S}) \cong \Zbb \to \Sbb^0(BC_2) \cong \widehat{A}(C_2) \cong \Zbb \{1\} \oplus \Zbb_2\{y\}. \]
Here, $\widehat{A}(C_2)$ is the completion of the Burnside ring $A(C_2)$ at the augmentation ideal. The second-to-last isomorphism is due to the Segal conjecture (the case for $C_2$ is proven in \cite{Lin_1980}, and the general case is proven in \cite{8ce2daad-bfb3-3964-bf77-9771ee1483d8}), which in particular tells us the natural map $A(C_2) \to \Sbb^0(BC_2)$ is isomorphic to the completion at the augmentation ideal.

\begin{lemma}
   Under the generators described above, $p_{!}^{\mathbb{S}}(1) = 2 + y$.
\end{lemma}

\begin{proof}
    The proof is somewhat tautological. From construction, the element $p_{!}^{\mathbb{S}}(1)$ represents the transfer map $\operatorname{tr}: \Sigma^{\infty}_+ BC_2 \to \mathbb{S} \in \mathbb{S}^0(BC_2)$. The canonical map $A(C_2) \to \Sbb^0(BC_2)$ sends the element $[C_2/e] \in A(C_2)$ to the transfer class $\operatorname{tr}$. The augmentation ideal of $A(C_2)$ is generated by $\operatorname{tr} - 2$, which we wrote as $y$.
\end{proof}

The inclusion of subcomplex $i: \mathbb{R}P^2 \hookrightarrow BC_2$ induces a class $i^*(2+y) \in \Sbb^0(\Rbb P^2)$. The unit map $\varphi: \Sbb \to \operatorname{KO}$ gives a ring map $\mathbb{S}^0(\mathbb{R}P^2) \to \operatorname{KO}^0(\mathbb{R}P^2)$ and determines a class
\[\varphi_*(i^*(2+y)) \in \operatorname{KO}^{0}(\Rbb P^2).\]

\begin{lemma}\label{lem::y_calc}
    Let $L$ be the tautological line bundle over $\Rbb P^2$. We have that $\varphi_*(i^*(2+y)) = [L] + 1 \in \operatorname{KO}^0(\Rbb P^2)$. and it follows that $\varphi_*(i^*(y)) = [L] - 1$.
\end{lemma}

\begin{proof}
    By naturality, the class $\varphi_*(i^*(2+y))$ is actually the image of $1 \in \operatorname{KO}^0(S^2)$ (the trivial line bundle, which we will write $\epsilon^1$) under the additive transfer map $\operatorname{KO}^{0}(S^2) \to \operatorname{KO}^{0}(\mathbb{R}P^2)$.
    
    This transfer has a more classical geometric description by Atiyah \cite{Atiyah1961}, see also Section 7 of \cite{BECKER19751}. In general, for a finite-sheeted covering space $\pi: B' \to B$ and a real vector bundle $V$ over $B'$, one can form a real vector bundle $\pi_*(V)$ on $B$ whose fiber at $x \in B$ is $\bigoplus_{y \in \pi^{-1}(x)} V_{y}$. The resulting vector bundle on $B$ determines the map $\operatorname{KO}^0(B') \to \operatorname{KO}^0(B)$.

    In this specific case, we wish to show $\pi_*(\epsilon^1)$ is isomorphic to $\epsilon^1 \oplus L$. Indeed, for each equivalence class $[v] \in \Rbb P^2$, the fiberwise isomorphism $\pi_*(\epsilon^1)_{[v]} \to (\epsilon^1 \oplus L)_{[v]}$ is given by the well-defined map $(a, b) \mapsto (a + b, \frac{(a-b)}{2} v)$.
\end{proof}

As a corollary, we have that:
\begin{corollary}\label{cor::10y_not_zero}
  The class $\varphi_*(i^*(10 y))$ is not $0$ in $\operatorname{KO}^0(\Rbb P^2)$.
\end{corollary}

\begin{proof}
    By Lemma~\ref{lem::y_calc}, it suffices to show that $10([L] - 1) \neq 0$, so in other words that $10[L]$ is not a stably trivial bundle. Indeed, if $10[L]$ is a stably trivial bundle, then the total Stiefel-Whitney class $w(10[L]) = 1$, but from the Whitney product formula and the fact that $w([L]) = (1 + a)$, where $a$ is the generator of the ring $H^*(\Rbb P^2; \Fbb_2)$, we know that
    \[w(10[L]) = (1 + a)^{10} \pmod{2}.\]
    Since $a^n = 0$ for $n \geq 3$, we have that the equation becomes $1 + 10 a + 45 a^2$, which is $1 + a^2$ modulo $2$. Thus the total Stiefel-Whitney class of $10[L]$ is not $1$, so $10[L]$ is not stably trivial.
\end{proof}

\section{Setting Up the Proof}\label{sec::reduction}

The goal of this section is to explain why we can reduce Theorem~\ref{thm::main} to the following statement. The arguments used are standard and use the change-of-rings formula (Theorem~\ref{thm::change-of-rings}). 

\begin{theorem}\label{thm::no_e2_complete}
    Let $X$ be path-connected. There is no $\mathbb{E}_2$-map
    \[f: X \to \operatorname{BGL}_1(\mathbb{S}^{\wedge}_2)\]
    such that $Mf \cong H\Zbb/4$ as $\mathbb{E}_2$-$\mathbb{S}^{\wedge}_2$-algebras.
\end{theorem}

\begin{remark}
    Any (unital) $\mathbb{E}_2$-ring that is additively equivalent to $H\Zbb/4$ must be equivalent to $H\Zbb/4$ as $\mathbb{E}_2$-rings. This is because any unital commutative ring of order $4$ whose additive group is isomorphic to $\Zbb/4$ must also be isomorphic to $\Zbb/4$ as rings. The desired equivalence is supplied by the zeroth Postnikov truncation, which can be promoted to be $\mathbb{E}_2$. This means the statement in Theorem~\ref{thm::no_e2_complete} can be promoted to saying that there is no $\mathbb{E}_2$-map $f$ such that $Mf$ is even additively equivalent to $H\Zbb/4$.
\end{remark}




\subsection{Reducing To the Path-Connected $2$-Complete Case}

For now we suppose $H\Zbb/4$ is an $\mathbb{E}_2$-Thom spectrum over $\mathbb{S}_{(2)}$ arising from an $\mathbb{E}_2$-map:
\[\phi: X \to \operatorname{Pic}(\mathbb{S}_{(2)}).\]

Let us suppose that $M\phi$ is just additively equivalent to $H\Zbb/4$. Since colimits commute with coproducts, we can write
\[M\phi \simeq \bigvee_{Y \in \pi_0(X)} M(\phi|_{Y}).\]
On the other hand, by examining the homotopy groups, there are no non-trivial spectra that is a retract of $H\Zbb/4$ other than $H\Zbb/4$ and the zero spectrum. The unital map $X \to \operatorname{Pic}(\Sbb_{(2)})$ must give a non-zero summand of the Thom spectrum on the component of $X$ at identity. Thus, we can without loss assume that $X$ is path-connected. As an $\mathbb{E}_2$-map is unital, we now have a map:
\[\phi: X \to \operatorname{BGL}_1(\mathbb{S}_{(2)}).\]

Now we explain why it suffices to prove this over the $2$-complete sphere spectrum.

\begin{lemma}
    The $2$-completion map $\mathbb{S}_{(2)} \to \mathbb{S}^{\wedge}_2$ induces an $\mathbb{E}_2$-equivalence
    \[M(\underbrace{X \xrightarrow{\phi} \operatorname{BGL}_1(\mathbb{S}_{(2)}) \to \operatorname{BGL}_1(\mathbb{S}^{\wedge}_2}_{f}) \cong H\Zbb/4.\]
\end{lemma}

\begin{proof}
From Theorem~\ref{thm::change-of-rings}, we know there is an equivalence
  \[Mf \cong M\phi \otimes_{\Sbb_{(2)}} \Sbb^{\wedge}_{2} \cong H\Zbb/4 \otimes_{\Sbb_{(2)}} \Sbb_2^{\wedge}.\]
Tensoring the cofiber sequence $C \to \Sbb_{(2)} \to \Sbb_2^{\wedge}$ with $H\Zbb/4$ over ${\Sbb_{(2)}}$ gives the sequence
\[C \otimes_{\Sbb_{(2)}} H\Zbb/4 \to H\Zbb/4 \to Mf.\]
Now since $H\Zbb/4$ is $2$-local, we have that $$C \otimes_{\Sbb_{(2)}} H\Zbb/4 \cong C \otimes_{\Sbb_{(2)}} (\Sbb_{(2)} \otimes H\Zbb/4) \cong C \otimes H\Zbb/4. $$
The term $C \otimes H\Zbb/4$ is $0$ as $\Sbb_{(2)} \to \Sbb^{\wedge}_2$ is an $H\Fbb_2$-homology isomorphism (and hence an $H\Zbb/4$-homology isomorphism), since $2$-completion is the same as localizing by $H\Fbb_2$ for connective spectra.
\end{proof}

Combining the discussions in this section shows that Theorem~\ref{thm::no_e2_complete} implies Theorem~\ref{thm::main}.

\section{Constructing the class $\alpha$}\label{sec::construct_alpha}

In the remainder of the paper, we will work to prove Theorem~\ref{thm::no_e2_complete}. Throughout, we assume for contradiction that there is an $\mathbb{E}_2$-map
\[f: X \to \operatorname{BGL}_1(\mathbb{S}^{\wedge}_2)\]
such that $X$ is path-connected and $Mf \cong H\Zbb/4$ as $\mathbb{E}_2$-$\mathbb{S}^{\wedge}_2$-algebras.\\

In this section, we will explain how to construct a cohomology class $\alpha \in H^3(Y; \Fbb_2)$ (where $Y$ is a two-fold delooping of $X$) that will be used in the next section to derive a contradiction.

\begin{definition}
    We define the map $\rho$ by applying $\pi_1$ to $f$
    \[\rho: \pi_1(X) \to \pi_0 \operatorname{GL}_1(\mathbb{S}^{\wedge}_2) \cong \Zbb_2^{\times} \cong \{\pm 1\} \times (1 + 4 \Zbb_2).\]
\end{definition}

\begin{lemma}\label{lem::rho_non_zero}
    For each $g \in \pi_1(X)$, $\rho(g) \equiv 1 \pmod{4}$. Furthermore, there exists $g \in \pi_1(X)$ such that $\rho(g) \equiv 5 \pmod{8}$.
\end{lemma}

\begin{proof}
    Observe that since $\pi_0(X) = 0$, $X$ is a group-like $A_{\infty}$-space and $X \simeq B \Omega X$. Thus, Proposition~\ref{prop::quotient} implies that we can write $Mf$ as the homotopy orbit
    \[Mf \simeq (\Sbb^{\wedge}_{2})_{h\Omega X}.\]
    Now take $\pi_0$ on both sides. Since $\pi_0$ is a left adjoint on connective spectra and homotopy orbit is a colimit, we have that
    \[\pi_0(Mf) \cong (\pi_0 \Sbb^{\wedge}_{2})_{\pi_0 \Omega X} = (\Zbb_{2})_{\pi_1(X)},\]
    where $\pi_1(X)$ acts on $\Zbb_{2}$ via $\rho$. This amounts now to 
    \[\pi_0(Mf) = \Zbb_2/I_\rho, I_{\rho} = \langle \rho(g) - 1\ |\ g \in \pi_1(X)\rangle.\]
    Since every non-zero ideal of $\Zbb_2$ is of the form $\langle 2^m \rangle$ for some $m \geq 0$ and we know that $\Zbb_2/I_{\rho} \cong \Zbb/4$, it follows that there is an equality of ideals $I_{\rho} = \langle 4 \rangle$. Thus we necessarily have that $\rho(g) \equiv 1 \pmod{4}$ for all $g \in \pi_1(X)$. Furthermore, there has to be some $g$ such that $\rho(g) \equiv 5 \pmod{8}$ (or else they would all land in $\langle 8 \rangle$).
\end{proof}

\begin{definition}\label{def::chi}
    We define $\Lambda$ as the group homomorphism
    \[\Lambda: 1 + 4 \Zbb_2 \to \Fbb_2,\quad \Lambda(1 + 4 a) = a \pmod{2}.\]
    We define $\chi: \pi_1(X) \to \Fbb_2$ to be the composition
    \[\chi = \Lambda \circ \rho.\]
\end{definition}

\begin{lemma}
    There is an isomorphism $H^1(X; \Fbb_2) \cong \Fbb_2$, and $\chi$ represents the non-zero element in $H^1(X; \Fbb_2)$.
\end{lemma}

\begin{proof}
The reduction modulo $2$ map $\Zbb/4 \to \Zbb/2$ induces a Thom isomorphism (Theorem~\ref{thm::thom_iso}) between
\[H\Zbb/2 \otimes_{\mathbb{S}_{2}^{\wedge}} Mf \simeq H\Zbb/2 \otimes \Sigma_+^{\infty} X.\]
Since the map $\Sbb \to \Sbb^{\wedge}_2$ is a $H\Zbb/2$-homology isomorphism, we can drop the subscript $\Sbb_{2}^{\wedge}$ and have that
\[H\Zbb/2 \otimes Mf \simeq H\Zbb/2 \otimes \Sigma_+^{\infty} X.\]
Since $Mf \simeq H\Zbb/4$, this shows that $H_i(X; \Fbb_2) \cong H_i(H\Zbb/4; \Zbb/2)$ for each $i$. Consider the sequence
\[H\Zbb/2 \xrightarrow{i} H\Zbb/4 \xrightarrow{q} H\Zbb/2\]
and its long exact sequence for homology in $\Fbb_2$ coefficient. The connecting map $H_{i+1}(H\Zbb/2; \Fbb_2) \to H_i(H\Zbb/2; \Fbb_2)$ is the homological Bockstein and hence zero when $i = 1$ (as $\operatorname{Sq}^1 \operatorname{Sq}^1 = 0$). Since $H_0(H\Zbb/4; \Fbb_2) = \Fbb_2$, the map landing in it is $0$, and we have
\[0 \to H_1(H\Zbb/2; \Fbb_2) \to H_1(H\Zbb/4; \Fbb_2) \to H_1(H\Zbb/2; \Fbb_2) \xrightarrow{\cong} \Fbb_2 \to 0. \]
This shows that $H_1(X; \Fbb_2) \cong H_1(H\Zbb/4; \Fbb_2) \cong \Fbb_2$. Since each degree of $H_*(X; \Fbb_2) \cong H_*(H\Zbb/4; \Fbb_2)$ is clearly finitely generated, it follows that $H^1(X; \Fbb_2)$ is the dual of $H_1(X; \Fbb_2)$ and is also $\Fbb_2$.\\

Finally, to show that $\chi \in H^1(X; \Fbb_2)$ is non-zero, it suffices to show the map is non-zero, which follows directly from Lemma~\ref{lem::rho_non_zero}.
\end{proof}

Recall the Kronecker pairing gives a map $\langle -, - \rangle: H^k(X; R) \times H_k(X; R) \to R$. 
\begin{definition}\label{def::chi_x}
    We choose $x \in H_1(X; \Fbb_2)$ as a homology class such that $\langle \chi, x \rangle = 1$ under the Kronecker pairing.
\end{definition}

In what follows, we will construct a class $\alpha$ using the fact that $f$ is $\mathbb{E}_2$.
\begin{definition}\label{def::y}
    Since $f$ is $\mathbb{E}_2$, we can write $X \simeq \Omega^2 Y$ for some $2$-connected $Y$ and find a map 
    \[g: Y \to B^3 \operatorname{GL}_1(\mathbb{S}^{\wedge}_2)\]
    such that $f$ deloops twice to $g$. For convenience, we also write $X \simeq \Omega Z$ for some 1-connected $Z$ and a map $q: Z \to B^2\operatorname{GL}_1(\mathbb{S}_2^{\wedge})$ such that $\Omega g = q$ and $\Omega q = f$.
\end{definition}

We define the class $\alpha$ as follows:
\begin{definition}
    We construct the class $\alpha$ as follows.
    \begin{enumerate}
        \item We write $\rho_3: B^3 \operatorname{GL}_1(\mathbb{S}^{\wedge}_2) \to K(\Zbb_2^{\times}, 3)$ to denote the third Postnikov truncation map. Note that this is $\Omega^{\infty}$ of the spectrum-level Postnikov truncation from $\Sigma^{3} \mathfrak{gl}_1(\mathbb{S}^{\wedge}_2) \to \Sigma^{3} H(\mathbb{Z}^{\times}_2)$.
        \item Using the decomposition $\Zbb^{\times} \cong \{\pm 1\} \times (1 + 4 \Zbb_2)$ and noting that $U^+ \coloneqq 1 + 4 \Zbb_2$ is generated by $5^n$ for $n \in \Zbb_2$, we write $d: \Zbb_2^{\times} \to \mathbb{F}_2$ as the map $d((-1)^{\epsilon} (5)^m) = m \pmod{2}$.
        \item We define $\alpha$ to be the map
        \[Y \xrightarrow{g} B^3\operatorname{GL}_1(\mathbb{S}^{\wedge}_2) \xrightarrow{\rho_3} K(\Zbb_2^{\times},3) \xrightarrow{K(d, 3)} K(\Fbb_2, 3).\]
        This represents an element
        \[\alpha \in H^3(Y; \Fbb_2).\]
    \end{enumerate}
\end{definition}

For our purposes, we also derive an equivalent formulation of $\alpha$ below.

\begin{definition}\label{def::E}
    We write $\lambda: \Zbb_2 \to U^+ \coloneqq 1+4\Zbb_2$ to be the isomorphism given by sending $n \mapsto 5^n$. We write $E$ to denote the connective spectrum given in the following (homotopy) pullback square
\[\begin{tikzcd}[ampersand replacement=\&]
	E \& {\mathfrak{gl}_1(\mathbb{S}^{\wedge}_2)} \\
	{H\mathbb{Z}_2} \& {H\Zbb_2^\times}
	\arrow[from=1-1, to=1-2]
	\arrow[from=1-1, to=2-1]
	\arrow["{\text{0th Postnikov truncation}}", from=1-2, to=2-2]
	\arrow["{\text{inclusion} \circ \lambda}"', from=2-1, to=2-2]
\end{tikzcd}.\]
We write $T^+ = \Omega^{\infty} \Sigma^3 E$. Since $\Sigma$ is an equivalence and $\Omega^{\infty}$ is right adjoint, we can write $T^{+}$ as the homotopy pullback:
\[\begin{tikzcd}[ampersand replacement=\&]
	{T^+} \& {B^3 \operatorname{GL}_1(\mathbb{S}^{\wedge}_2)} \\
	{K(\mathbb{Z}_2, 3)} \& {K(\mathbb{Z}_2^{\times}, 3)}
	\arrow[from=1-1, to=1-2]
	\arrow[from=1-1, to=2-1]
	\arrow["{\rho_3}", from=1-2, to=2-2]
	\arrow[from=2-1, to=2-2]
\end{tikzcd}.\]
\end{definition}

Now consider the maps
\begin{equation}\label{eq::Y}
\begin{tikzcd}[ampersand replacement=\&]
	Y \&\& \\
	\& {T^+} \& {B^3 \operatorname{GL}_1(\mathbb{S}^{\wedge}_2)} \\
	{K(\pi_3(Y), 3)} \& {K(\mathbb{Z}_2, 3)} \& {K(\mathbb{Z}_2^{\times}, 3)}
	\arrow["g", from=1-1, to=2-3]
	\arrow[from=1-1, to=3-1]
	\arrow[from=2-2, to=2-3]
	\arrow[from=2-2, to=3-2]
	\arrow["{\rho_3}", from=2-3, to=3-3]
	\arrow["{\lambda^{-1} \circ \rho}"', from=3-1, to=3-2]
	\arrow["{\text{inclusion} \circ \lambda}"', from=3-2, to=3-3]
\end{tikzcd}
\end{equation}

\begin{lemma}\label{lem::y_t_plus}
    This diagram in (\ref{eq::Y}) commutes up to homotopy and hence induces a lift $g^+: Y \to T^+$ which is unique up to homotopy.
\end{lemma}

\begin{proof}
Since $Y$ is $2$-connected, the Hurewicz theorem and the universal coefficient theorem imply that
\[H^3(Y; \Zbb_2^{\times}) \cong \operatorname{Hom}(\pi_3(Y), \Zbb_2^{\times}) \cong \operatorname{Hom}(\pi_1(X), \Zbb_2^{\times}).\]
It suffices to prove that they give the same cohomology class. Now clearly $\rho_3 \circ g$ gives the map induced on $\pi_1$ for $f: X \to \operatorname{BGL}_1(\mathbb{S}_2^{\wedge})$, and the other direction also clearly gives the same class.
\end{proof}

\begin{definition}\label{def::alpha_tilde}
    We define the class $\Tilde{\alpha} \in H^3(Y; \mathbb{Z}_2)$ to be the map
    \[\Tilde{\alpha}: Y \xrightarrow{g^+} T^+ \to K(\mathbb{Z}_2, 3).\]
\end{definition}

\begin{proposition}\label{prop::reduction}
    Let $r_2: \mathbb{Z}_2 \to \mathbb{F}_2$ denote the reduction modulo $2$ map. Then \[(r_2)_*(\Tilde{\alpha}) = \alpha.\]
\end{proposition}

\begin{proof}
    This follows because the following diagram of groups commutes:
\[\begin{tikzcd}[ampersand replacement=\&]
	{\mathbb{Z}_2} \& {\mathbb{F}_2} \\
	{U^+} \& {\mathbb{Z}_2^{\times}}
	\arrow["{r_2}", from=1-1, to=1-2]
	\arrow["{n \mapsto 5^n}"', from=1-1, to=2-1]
	\arrow[hook, from=2-1, to=2-2]
	\arrow["d"', from=2-2, to=1-2]
\end{tikzcd}.\]
\end{proof}

By Lemma~\ref{lem::coh_map_loop}, we also know that:
\begin{lemma}
  We have the identification $(\sigma^*)^2(\alpha) = d \circ \rho = \chi$.
\end{lemma}

\begin{proof}
The equality $(\sigma^*)^2(\alpha) = d \circ \rho$ is clear. It remains to check that $\chi$, which is defined as $\Lambda \circ \rho$ in Definition~\ref{def::chi}, is equal to $d \circ \rho$, but $d$ and $\Lambda$ coincide on the image of $\rho$.
\end{proof}

\begin{definition}
    We define $u \coloneqq \sigma_* x \in H_2(Z; \mathbb{F}_2)$ and $z \coloneqq \sigma_* u \in H_3(Y; \mathbb{F}_2)$.
\end{definition}

As a corollary, we have that:
\begin{corollary}\label{cor::alpha_not_zero}
The class $\alpha$ is satisfies $\langle \alpha, z \rangle = 1$ and is hence non-zero in $H^3(Y; \mathbb{F}_2)$.
\end{corollary}

\begin{proof}
    We have by the Kronecker pairing that
    \[\langle \alpha, z \rangle = \langle \alpha, (\sigma_*)^2 x \rangle = \langle (\sigma^*)^2 \alpha, x \rangle = \langle \chi, x \rangle = 1,\]
    where the last equality follows from Definition~\ref{def::chi_x}.
\end{proof}

\section{The Mechanism of Failure}\label{sec::failure}

Now we will derive a contradiction by showing simultaneously that:
\begin{enumerate}
    \item $\alpha^2 \neq 0 \in H^6(Y; \mathbb{F}_2)$, in Theorem~\ref{thm::not_zero}.
    \item $\alpha^2 = 0 \in H^6(Y; \mathbb{F}_2)$ in Theorem~\ref{thm::alpha_2_zero}.
\end{enumerate}

\subsection{Non-Vanishing Case}

In this section, we will prove the following.
\begin{theorem}\label{thm::not_zero}
    The class $\alpha^2 \in H^6(Y; \mathbb{F}_2)$ is not zero.
\end{theorem}

To do this, we will prove two technical lemmas that will directly imply the theorem. We first make the following observation.
\begin{lemma}\label{lem::u2_not_zero}
  The class $u \in H_2(Z; \Fbb_2)$ satisfies $u^2 = 0$. Here, the product is taken with respect to the Pontryagin product.
\end{lemma}

\begin{proof}
    By Theorem~\ref{thm::DL_operations}, $u^2$ coincides with the Dyer-Lashof operations $Q^2(u) = Q^2(\sigma_*(x)) = \sigma_*(Q^2(x))$. On the other hand, since $f$ is $\mathbb{E}_2$, we know there is an $\mathbb{E}_2$-equivalence $H\Fbb_2 \otimes H\Zbb/4 \simeq H\Fbb_2 \otimes \Sigma_+^{\infty} X$, so $H_*(X; \mathbb{F}_2)$ has the same Dyer-Lashof operations up to $Q^2$ as $H_*(H\Zbb/4; \Fbb_2)$. The latter has been computed in Corollary III.2.8 of \cite{h_infty_book} and has $Q^2(x) = 0$. This was done with respect to the $\mathbb{E}_{\infty}$-structure on $H\Zbb/4$ in \cite{h_infty_book}, but the conclusion applies due to Theorem~\ref{thm::DL_operations}.
\end{proof}

To prove Theorem~\ref{thm::not_zero}, we will employ the following lemma which is certainly known to experts. Similar statements have appeared as Theorem 2 of \cite{DrachmanKraines1971} in the language of differential graded algebras.
\begin{lemma}\label{lem::nonvanishing}
    Let $W$ be a path-connected well-based topological monoid. Suppose $v \in H_q(W; \Fbb_2)$ is a homology class such that $v^2 = 0$, and $a 
    \in H^{q+1}(BW; \Fbb_2)$ is a cohomology class such that $\langle a, \sigma_* v \rangle = 1$, then the cup product $a^2 \neq 0$.
\end{lemma}

\begin{proof}
Write $A = C_*(W; \mathbb{F}_2)$ to be equipped with the Pontryagin product and $BA$ be the normalized bar construction from $A$. This has the property that there is an isomorphism of coalgebras
\[\Phi: H_*(BA) \xrightarrow{\cong} H_*(BW; \Fbb_2).\]
Let $V \in C_{q}(W; \Fbb_2)$ be a representative of the homology class $v$. Since $v^2 = 0 \in H_{2q}(W; \Fbb_2)$, we can find an element $H \in C_{2q+1}(W; \Fbb_2)$ such that $\partial H = V^2$. Now consider the word
\[\gamma \coloneqq [V|V] + [H] \in (BA)_{2q+2}.\]
Let $D$ be the differential on $BA$, on an arbitrary word, we have
\[D[A_1| ... | A_s] = \sum_{i=1}^s [A_1|...|\partial A_i|...|A_s] + \sum_{i=1}^{s-1} [A_1|...|A_i A_{i+1}| ... A_s].\]
When we apply to the word $\gamma$, this has that
\[D \gamma = [\partial V|V]+[V|\partial V] + [V^2] + [\partial H] = [0|V] + [V|0] + [V^2] + [\partial H] = 2[V^2] = 0,\]
since we work in characteristic $2$.\\

This shows that $\gamma$ gives a homology class $[\gamma] \in H_*(BA)$. Now let $c \coloneqq \Phi([\gamma]) \in H_*(BW; \Fbb_2)$. The map $\Phi$ preserves the coproduct, and we know that $\Delta \gamma = 1 \otimes \gamma + [V] \otimes [V] + \gamma \otimes 1$. Thus, we have that
\[\Delta_* c = 1 \otimes c + (\sigma_* v) \otimes (\sigma_* v)  + c \otimes 1,\]
where we note that $\Phi([V]) = \sigma_*v$ as the edge homomorphism in the bar spectral sequence is homological suspension (for example, see Proposition 7.2 of \cite{moore}). Now we have that
\begin{align*}
    \langle a^2, c \rangle &= \langle \Delta^*(a \times a), c \rangle\\
    &= \langle a \times a, \Delta_* c \rangle\\
    &= \langle a \times a, 1 \times c + (\sigma_* v) \times (\sigma_* v) + c \times 1 \rangle\\
    &= \langle a \times a, 1 \times c \rangle + \langle a \times a, c \times 1\rangle + \langle a, \sigma_* v \rangle ^2\\
    &= \langle a, \sigma_* v \rangle ^2\\
    &= 1.
\end{align*}
Here $a \times a$ evaluates to $0$ on $1 \times c$ and $c \times 1$, due to bidegree reasons on the Künneth decomposition for $H_{2q+2}(BW \times BW; \Fbb_2)$. Thus, the class $a^2$ is non-zero.
\end{proof}

The proof of Theorem~\ref{thm::not_zero} now follows from the two previous lemmas.
\begin{proof}[Proof of Theorem~\ref{thm::not_zero}]
    The result follows from Lemma~\ref{lem::nonvanishing} by choosing $W \coloneqq Z$, which can be made into a topological monoid using Moore's loop space model, $v \coloneqq u$, and $a \coloneqq \alpha$, and using Lemma~\ref{lem::u2_not_zero} and Corollary~\ref{cor::alpha_not_zero}.
\end{proof}

\subsection{Vanishing Case}

In this section, we will prove the following.
\begin{theorem}\label{thm::alpha_2_zero}
    The class $\alpha^2 \in H^6(Y; \Fbb_2)$ is $0$.
\end{theorem}

The proof of Theorem~\ref{thm::alpha_2_zero} will follow from $k$-invariant computations on the space $Y$. For convenience, we will do our $k$-invariant computations on the spectrum $E$ (Definition~\ref{def::E}), and the results will carry over to $Y$.\\

We first make the following observation.
\begin{lemma}\label{lem::E_info}
    The pullback map $E \to \mathfrak{gl}_1(\mathbb{S}^{\wedge}_2)$ is an isomorphism on $\pi_{\geq 1}$. The first three homotopy groups of $E$ are
    \[\pi_0(E) = \Zbb_2, \pi_1(E) = \Zbb/2\{\psi_1(\eta)\}, \pi_2(E) = \Zbb/2,\]
    where $\eta$ corresponds to the stable Hopf map and $\psi_1$ is from Proposition~\ref{prop::eta_formula}. Furthermore, the map $\pi_0(E) \cong \Zbb_2 \to \pi_0(H\Zbb_2^{\times})$ in the pullback square is given by $n \mapsto 5^n$.
\end{lemma}

\begin{proof}
    This follows directly from the Mayer-Vietoris sequence for homotopy pullbacks and examining the two maps into $H\Zbb_2^{\times}$ for the pullback diagram on $\pi_0$. 
\end{proof}

The first Postnikov truncation of $E$ then fits in a sequence
\[\Sigma H\Fbb_2 \to \tau_{\leq 1} E \to H\Zbb_2 \xrightarrow{\kappa_1} \Sigma^2 H\Fbb_2.\]

\begin{lemma}\label{lem::k1_zero}
    The $k$-invariant $\kappa_1 \in H^2(H\Zbb_2; \Fbb_2)$ is $0$.
\end{lemma}

\begin{proof}
    We can apply the naturality of $k$-invariant (see Section 3 of \cite{Gurski_2017}, for example) to the pullback map $\psi: E \to \mathfrak{gl}_1(\mathbb{S}^{\wedge}_2)$ and an equality:
    \[(\Sigma^{2} H \pi_1 E \to \Sigma^2 H \pi_1 \mathfrak{gl}_1(\mathbb{S}^{\wedge}_2))\circ \kappa_1^{E} = \kappa_1^{\mathfrak{gl}_1(\mathbb{S}^{\wedge}_2)} \circ (\tau_{\leq 0} E \to \tau_{\leq 0} \mathfrak{gl}_1(\mathbb{S}^{\wedge}_2)).\]
    By Lemma~\ref{lem::E_info}, we know $\psi$ is an equivalence in $\pi_{\geq 1}$, so it suffices to show that $\kappa_1^{\mathfrak{gl}_1(\mathbb{S}^{\wedge}_2)} \circ (\tau_{\leq 0} E \to \tau_{\leq 0} \mathfrak{gl}_1(\mathbb{S}^{\wedge}_2))$ is $0$.\\

    Now \cite{Kitchloo2020} calculates $\kappa_1^{\mathfrak{gl}_1(\mathbb{S}^{\wedge}_2)}$ to be the map $\operatorname{Sq}^2 \circ H(\operatorname{sgn})$, where $\operatorname{sgn}: \Zbb_2^{\times} \to \Fbb_2$ is the projection onto its sign. The map $\tau_{\leq 0} E \to \tau_{\leq 0} \mathfrak{gl}_1(\mathbb{S}^{\wedge}_2)$ is exactly the inclusion $\lambda: \Zbb_2 \to \Zbb_2^{\times}$ by $n \mapsto 5^n$. In particular, this means that $\operatorname{sgn} \circ \lambda$ is the zero map.
\end{proof}

Thus, we have a splitting $\tau_{\leq 1} E \simeq H\Zbb_2 \vee \Sigma H\Fbb_2$. This enables us to consider the next $k$-invariant.
\[\Sigma^2 H\Fbb_2 \to \tau_{\leq 2} E \to \tau_{\leq 1} E \simeq H \Zbb_2 \vee \Sigma H\Fbb_2 \xrightarrow{\kappa_2} \Sigma^3 H\Fbb_2.\]

Using our knowledge of Steenrod squares, we can write 
\[\kappa_2 \in H^3(H\Zbb_2; \Fbb_2) \oplus H^2(H\Fbb_2; \Fbb_2) \cong \Fbb_2 \{\operatorname{Sq}^3 \circ r_2\} \oplus \Fbb_2 \{\operatorname{Sq}^2\},\]
where $r_2: \Zbb_2 \to \Fbb_2$ denotes reduction modulo $2$.

\begin{definition}
In terms of the basis above, we can write
\[\kappa_2 = a_1 (\operatorname{Sq}^3 \circ r_2) + a_2 (\operatorname{Sq}^2), a_i \in \Fbb_2.\]    
\end{definition}

\begin{lemma}\label{lem::a2_zero}
    The coefficient $a_2$ is equal to $0$.
\end{lemma}

\begin{proof}
    Suppose $a_2 = 1$ and recall that $\operatorname{Sq}^2$ corresponds to a non-zero action of $\eta$. It suffices for us to show that $\eta$ acts trivially from $\pi_1(E) \to \pi_2(E)$. Since the map $E \to \mathfrak{gl}_1(\mathbb{S}^{\wedge}_2)$ is an isomorphism on $\pi_{\geq 1}$, it suffices to show $\eta$ acts trivially from $\pi_1(\mathfrak{gl}_1(\mathbb{S}^{\wedge}_2)) \to \pi_2(\mathfrak{gl}_1(\mathbb{S}^{\wedge}_2))$.\\
    
    Indeed, by Proposition~\ref{prop::eta_formula}, write $\psi_1: \pi_1(\mathbb{S}^{\wedge}_2) = \Zbb/2\{\eta\} \to \pi_1(\mathfrak{gl}_1(\mathbb{S}^{\wedge}_2))$ to be the canonical isomorphism, then
    \[\eta \cdot \psi_1(\eta) = \psi_2(\eta^2 + \eta^2) = \psi_2(2 \eta^2) = \psi_2(0) = 0.\]
    It follows that $a_2$ has to be $0$.
\end{proof}

\begin{lemma}
    The coefficient $a_1$ is equal to $1$.
\end{lemma}

\begin{proof}
    Write $R = \mathbb{S}^{\wedge}_2$. For the ease of notation, we will write $F = \tau_{\leq 2} E$ and $G = \tau_{\leq 2} \mathfrak{gl}_1(R)$. Suppose for contradiction that $a_1 = 0$, then from Lemma~\ref{lem::a2_zero}, we know that $\kappa_2 = 0$. This then gives a section
    \[s: H\Zbb_2 \to F \text{ and we define } h: H\Zbb_2 \xrightarrow{s} F \to G,\]
    where the second map is induced by the pullback map $E \to \mathfrak{gl}_1(R)$. Lemma~\ref{lem::E_info} implies that, after taking $\pi_0$, we have $h_*(1) = 5$ and $h_*(2) = 5^2 = 25$.\\

    Now let $\operatorname{tr}: \Sigma^{\infty}_+ BC_2 \to \mathbb{S}$ denote the transfer map (Definition~\ref{def::transfer}) corresponding to the double cover $EC_2 \to BC_2$. Note that since this is a double cover, we have that $p^{H\Zbb_2}_!(1) = \underline{2}$. By naturality, we have that
    \[p^{G}_!(5) = p^{G}_{!}(h_*(1)) = h_*(p_!^{H\Zbb_2}(1)) = h_*(\underline{2}) = \underline{25}.\]
    Here, $\underline{n}$ is the element in cohomology of $BC_2$ induced by pulling back the number $n$ along the map $BC_2 \to *$.\\

    Now by construction $p_{!}^G(5) = \underline{25}$ lives in $G^0(BC_2)$. Consider the inclusion of the subcomplex $i: \Rbb P^2 \hookrightarrow BC_2$, this then yields a class
    \[i^*(p_{!}^{G}(5)) = i^*(\underline{25}) \in G^0(\Rbb P^2).\]
    Naturality shows that $i^*(\underline{25})$ is also the pullback of the constant class $25$ in $G^0(\Rbb P^2)$, so with a slight abuse of notation,
    \[i^*(p_{!}^{G}(5)) = \underline{25} \in G^0(\Rbb P^2).\]

    Now we will derive a contradiction by showing that $i^*(p_{!}^{G}(5)) \neq \underline{25} \in G^0(\Rbb P^2)$. Recall that $G = \tau_{\leq 2} \mathfrak{gl}_1(\mathbb{S}^{\wedge}_2)$, so the natural map $\mathfrak{gl}_1(\mathbb{S}^{\wedge}_2) \to G$ is an isomorphism in $0$-th cohomology on $\Rbb P^2$. This is because $\Sigma^{\infty}_+ \Rbb P^2$ only has cells in degree $\leq 2$, and $\tau_{\geq 3} \mathfrak{gl}_1(\Sbb^{\wedge}_2)$ is $2$-connected. Therefore, by naturality it suffices to show that
    \[i^*(p_{!}^{\mathfrak{gl}_1(\mathbb{S}^{\wedge}_2)}(\underline{5})) \neq \underline{25} \in (\mathfrak{gl}_1\mathbb{S}^{\wedge}_2)^0(\Rbb P^2).\]
    Now $(\mathfrak{gl}_1(\mathbb{S}^{\wedge}_2))^0(\Rbb P^2)$ admits a natural injection 
    $$\iota: (\mathfrak{gl}_1(\mathbb{S}^{\wedge}_2))^0(\Rbb P^2) \to (\mathbb{S}^{\wedge}_2)^{0}(\Rbb P^2). $$
    The composition of this injection $\iota$ with the additive transfer on $\mathfrak{gl_1}(\mathbb{S}^{\wedge}_2)$ is the multiplicative transfer on $\mathbb{S}^{\wedge}_2$ (see Section 7 of \cite{15fd3722-af20-3529-b950-b5d19ca1d240}), as addition in the units is multiplication from the original ring. Thus, by naturality, it suffices to show that
    \[i^*(p_{\otimes}^{\mathbb{S}^{\wedge}_2}(\underline{5})) \neq \underline{25} \in (\mathbb{S}^{\wedge}_2)^0(\Rbb P^2).\]
    On the other hand, by Remark VIII.1.6.(i) of \cite{h_infty_book}, we can without loss write $p_{\otimes}^{\mathbb{S}^{\wedge}_2}(\underline{5}) = P_2(\underline{5})$, where $P_2$ is the power operation also in \cite{h_infty_book}. Proposition VIII.1.4 of \cite{h_infty_book} gives a formula that
    \[P_2(a + b) = P_2(a) + P_2(b) + p_!^{\Sbb^{\wedge}_2}(ab).\]
    Write $t = p_!^{R}(\underline{1})$. Since $p_!^{R}$ is additive, we know that $nt = p_!^{R}(\underline{n})$. Doing this recursively, we have a formula that $P_2(\underline{n}) = \underline{n} + {n \choose 2} t$, so we can find $P_2(\underline{5})$ to be
    \[P_2(\underline{5}) = \underline{5} + 10 t.\]
    Now write $y = t - \underline{2}$, we have that
    \[P_2(\underline{5}) = \underline{25} + 10 y.\]
    It then suffices to show that the restriction of $10y$ to $\Rbb P^2$ is non-zero (i.e., $i^*(10 y) \neq 0$). This now follows from Corollary~\ref{cor::10y_not_zero} and the fact that the map $\Sbb \to \Sbb^{\wedge}_2$ induces an isomorphism in the $0$-th reduced cohomology $\widetilde{\Sbb}^0(\Rbb P^2) \to \widetilde{\Sbb_2^{\wedge}}^0(\Rbb P^2)$.
\end{proof}

Putting the preceding lemmas together, we now have that:
\begin{corollary}\label{cor::k2}
The $k$-invariant $\kappa_2$ is the map
\[\kappa_2 = \operatorname{Sq}^3 \circ r_2.\]
\end{corollary}

Now we will finally prove Theorem~\ref{thm::alpha_2_zero}, that the class $\alpha^2 \in H^6(Y; \Fbb_2)$ is $0$.

\begin{proof}[Proof of Theorem~\ref{thm::alpha_2_zero}]
In what follows we write $P_j$ for the Postnikov truncations of spaces and $k_j$ for the $k$-invariant. Now recall $T^+ = \Omega^{\infty} \Sigma^3 E$ fitting in the homotopy pullback
    \[\begin{tikzcd}[ampersand replacement=\&]
	{T^+} \& {B^3 \operatorname{GL}_1(\mathbb{S}^{\wedge}_2)} \\
	{K(\mathbb{Z}_2, 3)} \& {K(\mathbb{Z}_2^{\times}, 3)}
	\arrow[from=1-1, to=1-2]
	\arrow[from=1-1, to=2-1]
	\arrow["{\rho_3}", from=1-2, to=2-2]
	\arrow[from=2-1, to=2-2]
\end{tikzcd}.\]
Since $\kappa_1 = 0$ by Lemma~\ref{lem::k1_zero}, we know there is a splitting 
\[P_4 T^+ \simeq K(\Zbb_2, 3) \times K(\Fbb_2, 4).\]
The $5$-th Postnikov truncation fits in
\[K(\Fbb_2, 5) \to P_5 T^+ \to P_4 T^+ \xrightarrow{k_5} K(\Fbb_2, 6).\]
From Corollary~\ref{cor::k2}, we know that $k_5$ is the map $\operatorname{Sq}^3 \circ r_2$\\

Now recall from Lemma~\ref{lem::y_t_plus} there is a map $g^+: Y \to T^+$. We write $g_4$ to denote the composition of $g^+$ with the $4$-th Postnikov truncation to $P_4 T^+$. Since $P_4 T^+$ splits, this by definition can be written in coordinates as $(\Tilde{\alpha}, \beta): Y \to K(\Zbb_2, 3) \times K(\Fbb_2, 4)$ (see Definition~\ref{def::alpha_tilde}). Since $k_5$ is $0$ on the second factor, we then have from Proposition~\ref{prop::reduction} that
\[g_4^* k_5 = \operatorname{Sq}^3(r_2 \circ \Tilde{\alpha}) = \operatorname{Sq}^3 \alpha = \alpha \cup \alpha,\]
where $\operatorname{Sq}^3 \alpha = \alpha \cup \alpha$ as the degree of $\alpha$ is $3$. On the other hand, the definition of the next $k$-invariant by definition compose to $0$ with the truncation map to the previous stage. Thus, we have that
\[0 = g_4^* k_5.\]
This shows that $\alpha \cup \alpha = 0$.
\end{proof}

Combining our previous discussions, we may now prove Theorem~\ref{thm::no_e2_complete}.
\begin{proof}[Proof of Theorem~\ref{thm::no_e2_complete}]
    Assume for contradiction that there is an $\mathbb{E}_2$-map $f: X \to \operatorname{BGL}_1(\mathbb{S}^{\wedge}_2)$ such that $Mf \cong H\Zbb/4$ as $\mathbb{E}_2$-$\mathbb{S}^{\wedge}_2$-algebras. In Section~\ref{sec::construct_alpha}, we constructed a cohomology class $\alpha$ associated to the data of $f$. However, Theorem~\ref{thm::not_zero} shows that $\alpha^2 \neq 0$, and Theorem~\ref{thm::alpha_2_zero} shows that $\alpha^2 = 0$. This causes a contradiction.
\end{proof}

\begin{remark}\label{rem::explain}
The proof that $H\Zbb/4$ is not an $\mathbb{E}_2$-Thom spectrum does not generalize to $H\Zbb/p^k$ for $(p, k) \neq (2, 2)$.
\begin{itemize}
    \item  When $p > 2$, the relevant groups $\pi_1(\mathfrak{gl}_1(\Sbb^{\wedge}_p))$ and $\pi_2(\mathfrak{gl}_1(\Sbb^{\wedge}_p))$ are both $0$, so the argument here collapses.
    \item When $p = 2$ and $k > 2$, the same argument as in Lemma~\ref{lem::rho_non_zero} would produce a map $\rho: \pi_1(X) \to \Zbb_2^{\times}$ that would land in $1 + 2^{k} \Zbb_2$ for $2^k > 4$. This $\rho$ would be $0$ after composing with $\Lambda$, so the argument here collapses.
    \item When $p = 2$ and $k = 1$, the action of $Q^2$ on the generator $x \in H_1(H\Zbb/2; \Fbb_2)$ is actually non-zero (in contrast to the proof of Lemma~\ref{lem::u2_not_zero}), so the argument collapses.
\end{itemize}
\end{remark}

\printbibliography

\end{document}